\documentclass[letterpaper]{article}
\usepackage[utf8]{inputenc}

\usepackage{amsmath}
\usepackage{amsfonts}
\usepackage{amssymb}
\usepackage{amsthm}
\usepackage{tikz}
\usepackage{enumerate}
\usepackage{xcolor}
\usepackage{graphicx}
\usepackage[final]{microtype}
\usetikzlibrary{arrows.meta,bending,positioning}

\theoremstyle{plain}
\newtheorem{thm}{Theorem}

\newtheorem{lem}[thm]{Lemma}

\theoremstyle{definition}

\theoremstyle{remark}

\DeclareMathOperator{\p}{\mathbb{P}}

\title{On the Last New Vertex Visited by a Random Walk in a Directed Graph}
\author{Calum~Buchanan \and Paul~Horn \and Puck~Rombach}
\date{\today}

\begin{document}
\maketitle

\begin{abstract}
Consider a simple graph in which a random walk begins at a given vertex. It moves at each step with equal probability to any neighbor of its current vertex, and ends when it has visited every vertex. We call such a random walk a \emph{random cover tour}. It is well known that cycles and complete graphs have the property that a random cover tour starting at any vertex is equally likely to end at any other vertex. Ronald Graham asked whether there are any other graphs with this property. In 1993, L\'aszlo Lov\'asz and Peter Winkler showed that cycles and complete graphs are the only undirected graphs with this property. We strengthen this result by showing that cycles and complete graphs (with all edges considered bidirected) are the only directed graphs with this property.

\smallskip


\smallskip

\end{abstract}

\medskip

Let $G$ be a connected digraph. A {\em cover tour} of $G$ from a vertex $u$ is a directed walk which begins at $u$ and ends when it has visited every vertex of $G$.
Lov\'asz and Winkler showed that cycles and complete graphs are the only undirected graphs with the property that a random cover tour beginning at any given vertex is equally likely to end at any other vertex~\cite{Lovasz1993}. We address a question posed by Winkler of whether there are any further directed graphs with this property.

We will borrow the notation used in~\cite{Lovasz1993}; let $L(u,v)$ be the event that $v$ is the last vertex to visited by a random cover tour beginning at $u$. In the case of a bidirected cycle, one might be tempted to believe that a neighbor of the starting vertex should be less likely to be last visited than one further from the start. While this intuition does not hold for cycles, it does hold for many other connected graphs. It is also shown in~\cite{Lovasz1993} that, for any two nonadjacent vertices $u$ and $v$ of a bidirected graph $G$, there is a neighbor $x$ of $u$ such that $\p(L(x,v)) \leq \p(L(u,v))$, and this inequality is strict if the induced subgraph $G - \{u,v\}$ is connected. From the proof, it follows that non-neighbors of a given vertex $u$ are at least as likely as neighbors to be last visited on a random walk from $u$. We generalize this notion in the following lemma (the proof is nearly identical to that of Theorem~2 in~\cite{Lovasz1993}).

\begin{lem}\label{lem:outneighbors}
Let $G$ be a connected digraph, and let $u$ be a vertex in $G$. If $u$ does not have an edge to a vertex $v$, then there is an out-neighbor $x$ of $u$ such that $\p(L(x,v)) \leq \p(L(u,v))$. Furthermore, this inequality is strict if there is a
directed walk from $u$ to $v$ which visits all vertices in $G$, but does not revisit $u$.
\end{lem}

\begin{proof}
Let $u$ and $v$ be vertices of a connected digraph $G$ such that $uv \not\in E(G)$. Let $x_1, \ldots, x_d$ denote the out-neighbors of $u$, and let $L(x_i;v,u)$ denote the event that $u$ is the last vertex, and $v$ the next-to-last vertex, in a random cover tour of $G$ beginning at $x_i$. It is not hard to see that the event $L(u,v)$ is the disjoint union of the events $L(x_i,v)$ and $L(x_i;v,u)$; either the random cover tour from $u$ to $v$ visits $u$ more than once, or it does not. Thus,
\begin{equation}\label{eq:mean}
	\p(L(u,v)) = \sum_{i=1}^d \frac{1}{d} \left(\p(L(x_i,v)) + \p(L(x_i;v,u))\right).
\end{equation}
It follows that $\p(L(u,v))$ is at least the mean of the $\p(L(x_i,v))$ over all out-neighbors $x_i$ of $u$.

The second statement also follows from equation~\eqref{eq:mean}. If it is possible for a cover tour to start at $u$ and end at $v$ without revisiting $u$, then $\p(L(x_i;v,u)) > 0$ for some $i \in \{1, \ldots, d\}$. In this case, $\p(L(u,v))$ is strictly larger than the mean of the $\p(L(x_i,v))$.
\end{proof}

In particular, if $T$ is a cover tour from $u$ to $v$ in a digraph $G$ with the property that a random cover tour from any given vertex is equally likely to end at any other vertex, then either $uv \in E(G)$ or $u$ appears at least twice in $T$.

\begin{thm}\label{thm:main}
In any directed graph with the property that $\p(L(u,v)) = \p(L(u,w))$ for any three distinct vertices $u$, $v$, and $w$, every edge is bidirected.
\end{thm}

\begin{proof}
Suppose, for the sake of contradiction, that there exists a pair of vertices $u$ and $v$ in a digraph $G = (V,E)$, as described, such that $vu \in E$ and $uv \not\in E$. By Lemma~\ref{lem:outneighbors}, any cover tour from $u$ to $v$ in $G$ revisits $u$. Let $T$ denote a shortest such cover tour, and let $T'$ denote the closed walk $T+vu$.

We begin by considering the cover tour which starts at $v$, first takes the edge $vu$, and continues along $T$ until the last seen vertex, $v_1$. The vertex $v$ appears only once in this cover tour, since $v$ appears only once in $T$. By Lemma~\ref{lem:outneighbors}, we have $vv_1 \in E$. Notice also that $v_1$ appears only once in $T$. Otherwise, the walk that starts at $u$, follows $T$ to the first copy of $v_1$ (visiting every vertex but $v$), and then follows $T$ from the last copy of $v_1$ back to $v$ is a cover tour that is shorter than $T$, which contradicts its minimality.

Now consider the cover tour contained in $T'$ which starts at $v_1$ and ends at the last seen vertex, $v_2$. As before, we have $v_1v_2 \in E$, and $v_2$ appears only once in $T$. We can continue in this way until we have chosen all of the vertices in $G$ which appear only once in $T$, for $T'$ is a closed walk containing all of the vertices of $G$, and each vertex which appears only once will be the last seen vertex on a walk from some $v_i$. Let $U = \{v, v_1, \ldots, v_{k}\}$ denote this set of chosen vertices, and let $W$ denote the closed walk that visits them in order. Note that $W$ visits the vertices in $U$ in opposite order to $T$. (See Figure~\ref{fig:closedwalk}.)

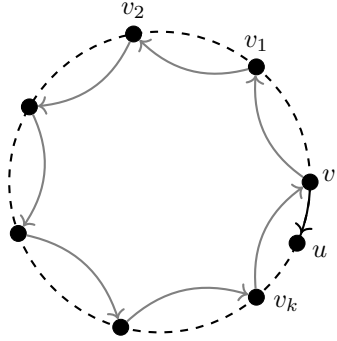
\begin{figure}
	\centering
	\begin{tikzpicture}
[every node/.style={circle, draw=black!100, fill=black!100, inner sep=0pt, minimum size=6pt}, scale=2]

	\draw[dashed, thick] (0,0) circle (1);
	
	\node (v) at (360:1) [label={[xshift=2.5mm,yshift=-1.5mm]$v$}] {};
%	\node (x) at (20:1) [label={[xshift=1mm]$x$}] {};
	\node (v1) at (50:1) [label={$v_1$}] {};
	\node (v2) at (100:1) [label={$v_2$}] {};
	\node (v3) at (150:1) {};
	\node (v4) at (200:1) {};
	\node (v5) at (255:1) {};
	\node (vk) at (310:1) [label={[xshift=4mm,yshift=-4mm]$v_k$}] {};
	\node (u) at (336:1) [label={[xshift=3mm,yshift=-3.5mm]$u$}] {};

	\draw[->, gray!100, thick] (v) edge [bend left] (v1) {};
	\draw[->, gray!100, thick] (v1) edge [bend left] (v2) {};
	\draw[->, gray!100, thick] (v2) edge [bend left] (v3) {};
	\draw[->, gray!100, thick] (v3) edge [bend left] (v4) {};
	\draw[->, gray!100, thick] (v4) edge [bend left] (v5) {};
	\draw[->, gray!100, thick] (v5) edge [bend left] (vk) {};
	\draw[->, gray!100, thick] (vk) edge [bend left] (v) {};
	\draw[<-,thick] (340:1) arc (340:360:1);

	\end{tikzpicture}

	\caption{A minimum cover tour $T$ from $u$ to $v$ is depicted by a dashed line, moving in clockwise direction. The edge $vu$ is included to depict the closed walk $T' = T+vu$. The grey arrows indicate the closed walk $W$, which visits all vertices in $U$, in counterclockwise direction.}
	\label{fig:closedwalk}
\end{figure}

We note that any vertex $x$ in $V \setminus U$ has a copy in $T$, and the minimality of $T$ ensures that no two copies of $x$ appear in the same interval $(v_1, v)$, $(v, v_k)$, or $(v_{i+1}, v_{i})$ for $i \in \{1, \ldots, k-1\}$.
We will now show that $v_1v \in E$ and $v_{i+1}v_i \in E$ for each $i \in \{1, \ldots,k-1\}$. However, the minimality of $T$ then implies that there are no vertices between $v_1$ and $v$ in $T$; that is, $v_1v$ is an edge in $T$. Similarly, each edge $v_{i+1}v_i$ is in $T$. From this we derive our final contradiction, for $u$ appears at least twice in $T$ by assumption, and the two copies cannot both appear in the interval $(v_k, v)$.

We will find a cover tour in $G$ from $v_1$ to $v$, which does not revisit $v_1$. By Lemma~\ref{lem:outneighbors}, this implies that $v_1v \in E$. First, in $T'$ we label the $(v_2,v_1)$ interval $A$, the $(v_1,v)$ interval $B$, and the $(v,v_k)$ interval $C$. Since vertices in $V\setminus U$ appear multiple times on $T$, but only once in each interval, we use subscripts to differentiate copies. For example, $u_C$ indicates the vertex $u$ in the position on $T'$ in the interval $C$, as drawn in Figure~\ref{fig:closedwalk}. (We will only need labels for these three intervals.) Further, for any $x,y \in V$, with interval subscripts if necessary, we use the notation $x \overset{T'}{\rightarrow} y$ to indicate the walk from $x$ to $y$ along $T'$, and for $v_i,v_j$ in $U$ we use $v_i \overset{W}{\rightarrow} v_j$ to indicate the walk from $v_i$ to $v_j$ along $W$.

We now present a walk in two parts; the desired cover tour from $v_1$ to $v$ which does not revisit $v_1$ is a concatenation of these parts.
As a general template, we take Part (I) to be
\[ (v_1 \overset{W}{\rightarrow} v_k) \cup (v_k \overset{T'}{\rightarrow} v_2), \]
and Part (II) to be
\[ (v_2 \overset{W}{\rightarrow} v). \]
If there no vertices in $V \setminus U$ that appear only in the intervals $A$, $B$ or $C$ then the concatenation of Parts (I) and (II) yields the desired cover tour.
Otherwise, we modify Parts (I) and (II) depending on the intervals in which such vertices appear.

We label vertices in $V \setminus U$ which appear on $T'$ only in $A$, $B$, or $C$ as type $AB$, $BC$, or $AC$, depending on which two intervals they appear in. If a vertex appears in all three intervals, we may ignore one copy and label it arbitrarily. 
Part (I) of our walk starts at $v_1$ and ends at $v_2$, and it covers all vertices of type $BC$. If there are no $BC$ vertices, then Part (I) is as previously described.
Otherwise, let $x$ be the last $BC$ vertex in the interval $B$. In this case, Part (I) is 
\[ (v_1 \overset{T'}{\rightarrow} x_B) \cup (x_C \overset{T'}{\rightarrow} v_2) .  \]
Since $x$ was the last $BC$ vertex in $B$, Part (I) of the walk has now covered all $BC$ vertices.

Part (II) of our walk covers all $AB$ and $AC$ vertices. If there are none, then Part (II) is as previously described.
Otherwise, let $y$ be the last $AB/AC$ vertex in $A$. If $y$ is $AB$, then Part (II) is
\[ (v_2 \overset{T'}{\rightarrow} y_A) \cup (y_B \overset{T'}{\rightarrow} v) .  \]
If $y$ is $AC$, then Part (II) is
\[ (v_2 \overset{T'}{\rightarrow} y_A) \cup (y_C \overset{T'}{\rightarrow} v_k) \cup (v_k \overset{W}{\rightarrow} v).  \]
Since $y$ was the last unseen vertex in $A$, Part (II) of the walk has now covered all $AB$ and $AC$ vertices. Concatenating Part (I) and Part (II) yields a cover tour from $v_1$ to $v$ which does not revisit $v_1$, and therefore $vv_1 \in E$ (and $T$). A similar argument shows that $v_{i+1} v_i \in E(G)$ (and $T$) for all $i \in \{1, \ldots, k-1\}$, which completes the proof.
\end{proof}

\section*{Acknowledgements}
This work is the result of a collaboration at the 2021 Virtual Masamu Advanced Study Institute (MASI). Paul Horn was partially supported by Simons Collaboration grant 525039. We also thank Daniel Velleman for helpful comments.

\end{document}